\documentclass[12pt,reqno]{amsart}
\usepackage{amssymb,amsmath,amsthm}
\usepackage{graphicx}
\usepackage{subcaption}
\usepackage{fullpage}
\usepackage{scrextend}
\usepackage{siunitx}
\usepackage{enumerate}
\usepackage{tikz,calc}
\usepackage{tikz-cd}
\usetikzlibrary{automata,positioning}
\usepackage{comment}
\usetikzlibrary{calc}
\usepackage{epsfig,graphicx}
\usepackage{listings}
\usepackage{enumerate}
\usepackage{float}
\usepackage{bbm}

\newtheorem{theorem}{Theorem}[section]

\newtheorem{example}[theorem]{Example}
\newtheorem{conjecture}[theorem]{Conjecture}

\newtheorem{question}[theorem]{Question}

\theoremstyle{definition} 

\numberwithin{theorem}{section}

\newcommand{\RR}{\mathbb{R}}

\newcommand{\CC}{\mathbb{C}}

\newcommand{\QQ}{\mathbb{Q}}

\newcommand{\ZZ}{\mathbb{Z}}

\title{Sharp Szemer\' edi-Trotter constructions from arbitrary number fields}
\author{Gabriel Currier}

\begin{document}
\maketitle
\begin{abstract}In this note, we describe an infinite family of sharp Szemer\'{e}di-Trotter constructions. These constructions are cartesian products of arbitrarily high dimensional generalized arithmetic progressions (GAPs), where the bases for these GAPs come from arbitrary number fields over $\mathbb{Q}$. This can be seen as an extension of a recent result of Guth and Silier, who provided similar constructions based on the field $\mathbb{Q}(\sqrt{k})$ for square-free $k$. However, our argument borrows from an idea of Elekes, which produces cartesian products where the parts are of unequal size. This significantly simplifies the analysis and allows us to easily give constructions coming from any number field.
\end{abstract}

\section{Introduction} 

An \emph{incidence} between a pointset $\mathcal{P}$ and a set of curves $\mathcal{C}$ in the plane is a tuple $(p,c) \in \mathcal{P} \times \mathcal{C}$ such that the point $p$ is on the curve $c$. In 1983, Szemer\'edi and Trotter \cite{SZTR} gave a bound on the number of incidences between points and lines.

\begin{theorem}\label{szt1}
The maximum number of incidences between a set of $N$ points and a set of $M$ lines in $\RR^2$ is $O(N^{2/3}M^{2/3} + N + M)$\footnote{Here, we say $F = O(G)$ or equivalently $G = \Omega(F)$ if there exists a constant $C$ such that $F \le CG$.}.
\end{theorem}
\noindent This was extended to points and lines in $\CC^2$ in \cite{TO} and \cite{ZA}. Often, it is easier to use the following theorem, which can be shown to be equivalent to Theorem \ref{szt1}. This result is stated in terms of $r$-rich lines; that is, lines that contain at least $r$ points of the pointset.
\begin{theorem}\label{szt2}
A set of $N$ points in the plane determines at most $O(\frac{N^2}{r^3} + \frac{N}{r})$ $r$-rich lines.
\end{theorem}
\noindent These theorems are known to be tight. If $r > N^{1/2}$, in which case $\frac{N}{r}$ term dominates, then the constructions are trivial; one may simply take $r$ points on each of an arbitrary collection of $N/r$ lines. Historically, there were two examples known to match the bound when $r \le N^{1/2}$, where $\frac{N^2}{r^3}$ is the dominant term. Erd\H os showed that the $N^{1/2} \times N^{1/2}$ integer grid has $\Omega(\frac{N^2}{r^3})$ $r$-rich lines, and in \cite{EL2}, Elekes gave a simple argument showing the same for the $r \times N/r$ integer grid.
\\\\
\noindent For the last $20$ or so years, these have remained the only known constructions when $r \le N^{1/2}$; this left open the possibility that all sharp examples were integer grids, or simple transformations of them. However, very recently Guth and Silier \cite{GUSI} gave examples based on pointsets of the form $(a+b\sqrt{k},c+d\sqrt{k})$, where $0 \le |a|,|b|,|c|,|d| \le \frac{N^{1/4}}{2}$ and $k$ is a square-free integer. This is a $N^{1/2} \times N^{1/2}$ cartesian product and can be seen as analogous to Erd\H os' construction, but coming from the field $\QQ(\sqrt{k})$. The result we present here will instead generalize Elekes' construction. This gives a $r \times N/r$ cartesian product, but coming from \emph{any} finite dimensional extension of $\QQ$.
\\\\
To discuss this result, we need some definitions. We say that a field extension $K$ of $\QQ$ is a \emph{number field} if its dimension as a $\QQ$-vector space is finite. Furthermore, if $\alpha \in K$ is the root of a monic polynomial with integer coefficients, we say that $\alpha$ is an \emph{algebraic integer}. Now, if $\Lambda = \{\lambda_1,\dots,\lambda_n\}$ is a basis for $K$ over $\QQ$, we say that $\Lambda$ is a \emph{nice} basis if $\lambda_i\lambda_j$ is a $\ZZ$-linear combination of elements of $\Lambda$, for all $1\le i,j\le n$. One example of a nice basis is any integral basis - that is, a basis $\Lambda$ for $K$ such that any algebraic integer is a $\ZZ$-linear combination of the elements of $\Lambda$. Such a basis exists for any number field. Furthermore, if $[K:\QQ] = n$ and $\alpha \in K$ is an algebraic integer of degree $n$ over $\QQ$, then the \emph{power basis} $1,\alpha,\dots,\alpha^{n-1}$ is also a nice basis.
\\\\
Finally, for integers $k_1,k_2$, we let $[k_1,k_2] := \{k_1,k_1+1, \dots, k_2\}$. Furthermore, for a positive integer $m$, we define $A_m(\Lambda) := \{a_1\lambda_1 + \dots + a_n\lambda_n : |a_i| \in [0,\frac{m^{1/n}}{2}]\}$ and note that $|A_m(\Lambda)| = m$. We will generally omit floors and ceilings unless necessary for clarity.
\\\\
In this language, Elekes' construction shows that the $N$-element pointset $A_{r}(\{1\}) \times A_{N/r}(\{1\})$ determines $\Omega(N^2/r^3)$ $r$-rich lines for any $r \le N^{1/2}$. Analogously, we will prove the following result.

\begin{theorem}\label{mainthm}
Let $N,r$ be positive integers, with $r \le N^{1/2}$. Then, for any nice basis $\Lambda$ of a number field $K$ over $\QQ$, the set $A_r(\Lambda) \times A_{N/r}(\Lambda)$ determines $\Omega(N^2/r^3)$ \footnote{The implied constant here is allowed to depend on $K$ and $\Lambda$ but not $N$ or $r$} $r$-rich lines.
\end{theorem}
\noindent In particular, the pointsets generated here are cartesian products of arbitrarily high-dimensional generalized arithmetic progressions, and no such constructions were previously known. Furthermore, this gives a variety of sharp constructions in $\CC^2$. To our knowledge, prior to \cite{GUSI}, the only known sharp constructions in $\CC^2$ were simple transformations of constructions from $\RR^2$.
\\\\
\noindent It is worth taking a moment to discuss the motivation for these results. These constructions are part of a larger project to understand the \emph{inverse problem} for Szemer\' edi-Trotter; that is, attempting to classify configurations that can match the bound in Theorems \ref{szt1} and \ref{szt2} up to a constant. The Szemer\'edi-Trotter theorem and its relatives have seen numerous applications to a diverse range of problems (see e.g. \cite{BOBO,BODE,EL,ELRO} for some examples or \cite{DV,EL2, TAVU} for additional results) and have been instrumental in the development of discrete geometry and additive combinatorics. Thus, a reasonable characterization of the extremal configurations would immediately translate into information about (and potentially better bounds on) a diverse range of problems. In our view, understanding the possible constructions is an integral part of making progress on this difficult problem. For further references on the inverse problem, see for example \cite{EL3,EL5,KASI,PERORUWA,RUSH,SO}.
\\\\
\noindent As a second motivation, these constructions give new possibilities for pointsets in other incidence geometry problems. In general, many incidence geometry results are not known to be tight. The best lower bounds often come from finding an affine copy of some curve that intersects an integer grid as much as possible, and then simply taking many translates of that curve to create many incidences. It would be helpful to have other good candidates for pointsets to test when attempting to generate lower bounds on these problems. Furthermore, there is evidence that the pointsets described in Theorem \ref{mainthm} might be useful for this purpose, since they can generate many incidences with another polynomial curve (parabolas) as well. This is discussed in more depth in the section \ref{conc}.
\\\\
\noindent As a final motivation, for several other incidence-type problems, the best known lower bounds are obtained directly from sharp Szemer\'edi-Trotter constructions. In other words, to generate constructions for these other problems, we take a sharp Szemer\'edi-Trotter construction and perform a basic transformation on it, taking e.g. points to circles or lines to ellipses. See for example \cite{GOMOWH, SHSMVADE}. Thus, the constructions given by Theorem \ref{mainthm} translate immediately into new best-known constructions for other interesting problems.

\section{Construction}
For any nice basis $\Lambda = \{\lambda_1,\dots,\lambda_n\}$ of an $n$-dimensional number field $K$, let $c_k^{i,j}$ be the integers defined by $\lambda_i\lambda_j = c_1^{i,j}\lambda_1 + \dots + c_n^{i,j}\lambda_n$, and define $C_\Lambda := \max_{i,j,k} |c_k^{i,j}|$. To complete the proof of Theorem \ref{mainthm}, we will show that $A_r(\Lambda) \times A_{N/r}(\Lambda)$ has $\left(\frac{1}{(2n^2C_\Lambda)^n}\right)\frac{N^2}{r^3}$ $r$-rich lines.

\begin{proof}
Our line set $L$ will be defined as follows

$$L = \{y = mx + b : m \in A_{\frac{N}{(C_\Lambda n^2)^{n}r^2}}(\Lambda), b \in A_{\frac{N}{2^nr}}(\Lambda)\}.$$ 
We note that $|L| = \frac{N^2}{(2C_\Lambda n^2)^nr^3} = \Omega(\frac{N^2}{r^3})$. Thus, it remains to be shown only that each line is $r$-rich. Let $\ell \in L$ be the line $y = mx+b$ where $m = m_1\lambda_1 + \dots + m_n\lambda_n$ and $b = b_1\lambda_1 + \dots + b_n\lambda_n$, and let $x' = x_1\lambda_1 + \dots + x_n\lambda_n \in A_r(\Lambda)$. Then, we see 

\begin{align*}
y' &:=mx' +b = (m_1\lambda_1 + \dots + m_n\lambda_n)(x_1\lambda_1 + \dots + x_n\lambda_n) + (b_1\lambda_1 + \dots + b_n\lambda_n)\\
& = \sum_{1 \le i,j \le n} m_ix_j\lambda_i\lambda_j + (b_1\lambda_1 + \dots + b_n\lambda_n) \\
&= \sum_{1 \le k\le n}\left(\sum_{1\le i,j\le n} m_ix_jc_k^{i,j}+b_k\right)\lambda_k
\end{align*}
We note that for any $i,j,k$ we have $|m_ix_jc_k^{i,j}| \le \left(\frac{N^{1/n}}{2C_\Lambda n^2r^{2/n}}\right)\left(\frac{r^{1/n}}{2}\right)C_\Lambda = \frac{N^{1/n}}{4n^2r^{1/n}}$ and thus $\left|\sum_{1\le i,j\le n} m_ix_jc_k^{i,j}\right| \le \frac{N^{1/n}}{4r^{1/n}}$. Additionally, for each $k$ we have $$|b_k| \le \frac{N^{1/n}}{4r^{1/n}}$$ and so we conclude $$\left|\sum_{1\le i,j\le n} m_ix_jc_k^{i,j}+b_k\right| \le  \frac{N^{1/n}}{2r^{1/n}}$$ and thus $y' \in A_{N/r}(\Lambda)$. In particular, $(x',y') \in A_r(\Lambda) \times A_{N/r}(\Lambda)$ and since there are $r$ choices for $x'$, we conclude that $\ell$ is $r$-rich.
\end{proof}

\section{Concluding remarks}\label{conc}

For this final section we focus on the case $r \approx N^{1/3}$, as this produces roughly equal numbers of points and lines and is the most frequently discussed in the literature. However, we note that the ideas here should apply to any $r \le N^{1/2}$.
\\\\
In \cite{SHSI}, they show that one can ``interpolate'' between the integer grid constructions given by Erd\H os' ($N^{1/2} \times N^{1/2}$) and Elekes' ($N^{1/3} \times N^{2/3}$). That is, for any $1/3 < \alpha < 1/2$, the $N^{\alpha} \times N^{1-\alpha}$ integer grid determines $\Omega(N)$ $N^{1/3}$-rich lines, which by Theorem \ref{szt2} is maximum possible. We suspect that in general something similar is possible. We make the following conjecture which, if true, would include all known constructions.
\begin{conjecture}\label{conj1}
For any number field $K$, nice basis $\Lambda$, and real number $1/3 \le \alpha \le 1/2$, the set $A_{N^{\alpha}}(\Lambda) \times A_{N^{1-\alpha}}(\Lambda)$ determines $\Omega(N)$ $N^{1/3}$-rich lines.
\end{conjecture}

\noindent We've shown here the case that $\alpha = 1/3$, and in \cite{GUSI}, Guth and Silier showed the case when $\alpha = 1/2$ and $\Lambda = \{1,\sqrt{k}\}$. However, even in the limited case shown in \cite{GUSI}, the analysis is somewhat involved, and the difficulty of extending their argument appears to increase with the degree of the number field. Thus, it would be interesting to see if there is a way to simplify their analysis to show, at least, the case of $\alpha = 1/2$ for some other number fields. Furthermore, since conjecture \ref{conj1} covers all known constructions, it would of course be of particular interest to identify \emph{any} sharp construction that is not a simple transformation of one of these pointsets.
\\\\\
\noindent As mentioned in the introduction, there is also some evidence that these pointsets might serve as useful examples for other incidence problems.
\begin{example}
Let $\Lambda$ be a nice basis for a number field. We observe first that the parabola $y=x^2$ has $\Omega(N^{1/3})$ incidences with the pointset $A_{N^{1/3}}(\Lambda) \times A_{N^{2/3}}(\Lambda)$. To see this, we need note only that a positive fraction of squares from $A_{N^{1/3}}(\Lambda)$ are in $A_{N^{2/3}}(\Lambda)$. Taking translates of the parabola along vectors from $A_{N^{1/3}}(\Lambda) \times A_{N^{2/3}}(\Lambda)$, we get a collection of $\Omega(N)$ parabolas and $N$ points that determine $\Omega(N^{4/3})$ incidences. This is the maximum-possible number of incidences between $N$ points and $N$ translates of a fixed parabola, by Szemer\'edi-Trotter-type arguments (see e.g. \cite{PASH})
\end{example} 
\noindent What is not known, however, is how these pointsets behave with respect to other polynomial curves. As a specific example, the famous Erd\H os Unit Distance Problem asks for the maximum number of incidences between $N$ points and $N$ unit circles in $\RR^2$. Thus, the following would be of particular interest.
\begin{question}
What is the maximum possible number of incidences between a unit circle and an affine copy of the set $A_{N^{1/2}}(\Lambda) \times A_{N^{1/2}}(\Lambda)$?
\end{question}

\noindent \textbf{Acknowledgements:} I would like to thank Larry Guth, Simone Maletto, Olivine Silier, J\'ozsef Solymosi, Ethan White, and Chi Hoi Yip for helpful discussions and comments. The author is supported in part by a Killam Doctoral Scholarship from the University of British Columbia.

\end{document}